\documentclass[11pt,hyp]{amsart}
\usepackage{amsmath,amsthm,amsfonts,amssymb,bm,graphicx,enumitem,doi }

\usepackage{geometry}
\usepackage{hyperref}
\hypersetup{nesting=true,debug=true,naturalnames=true}
\usepackage{graphicx,amssymb,upref}

\title[Sums of squares in $S$-integers]{Sums of squares in $S$-integers}

\author{V\' \i t\v ezslav Kala}  
\address{Charles University, Faculty of Mathematics and Physics, Department of Algebra, Sokolov\-sk\' a 83, 18600 Praha~8, Czech Republic}
\email{kala@karlin.mff.cuni.cz}  

\author{Pavlo Yatsyna}
\address{Charles University, Faculty of Mathematics and Physics, Department of Algebra, Sokolov\-sk\' a 83, 18600 Praha~8, Czech Republic}
\email{yatsyna@karlin.mff.cuni.cz}  
\thanks{The authors were supported by projects PRIMUS/20/SCI/002 (V.K., P.Y.) and UNCE/SCI/022 (V.K.) from Charles University} 



\theoremstyle{plain}
\newtheorem{thm}{Theorem}
\newtheorem*{thm*}{Theorem}
\newtheorem{lemma}[thm]{Lemma}
\newtheorem{cor}[thm]{Corollary}

\newtheorem{innercustomthm}{Theorem}
\newenvironment{customthm}[1]
{\renewcommand\theinnercustomthm{#1}\innercustomthm}
{\endinnercustomthm}

\newtheorem{innercustomcor}{Corollary}
\newenvironment{customcor}[1]
{\renewcommand\theinnercustomcor{#1}\innercustomcor}
{\endinnercustomcor}

\theoremstyle{definition}

\newcommand{\co}{\mathcal O}
\newcommand{\Tr}{\mathrm{Tr}}

\begin{document}
 
\begin{abstract}   
      In totally real number fields, we characterize the rings of $S$-integers (obtained by inverting a rational integer $m$) such that all their totally positive elements are represented as a sum of squares. We further obtain partial answers to the question: when are all the totally positive algebraic integers that are divisible by $m$ represented as a sum of squares?
\end{abstract} 
\maketitle
%\tableofcontents

\section{Introduction}
In 1921, Siegel \cite{S1} proved Hilbert's conjecture that
\textit{in any number field $K$, every totally positive element can be represented as a sum of four squares of elements of $K$}.

The study of similar questions for rings of integers in number fields turned out to be much subtler. Of course, already in 1770 Lagrange proved that every positive integer is a sum of four squares.
In the case of totally real number fields, Maa\ss\ and Siegel \cite{Ma,S2} showed that the only other case when every totally positive integer is a sum of squares (in the ring of integers) is $\mathbb Q(\sqrt 5)$. In the non-totally real case, the necessary and sufficient condition for this to happen is that 2 is unramified \cite{S2}. Overall, the non-totally real case is better understood, and so we do not consider it in this note.

These results were then generalized in two main directions: Let the \textit{Pythagoras number} of a ring $R$ be the smallest positive integer $r$ such that every sum of squares in $R$ is also a sum of at most $r$ squares. In 1980 Scharlau \cite{Sc2} showed that Pythagoras numbers of rings of integers in totally real number fields $K$ can be arbitrarily large.
He also formulated a question whether these Pythagoras numbers are bounded in the case of bounded degree $[K:\mathbb Q]$; 
the authors of this note recently answered Scharlau's question in the affirmative
\cite{KY}.
For more results and history see, e.g., the survey by Leep \cite{Le}.

The second rich source of generalizations are \textit{universal quadratic forms}, i.e., quadratic forms that represent all totally positive integers. There is a rich literature concerning them, so let us only briefly note that they exist over each totally real number field by the results of \cite{HKK}, but can require arbitrarily large numbers of variables \cite{BK,Ka,Ya}.

\medskip

In this short note, we will focus on the intermediate case between number fields and their rings of integers $\co$, namely, on the rings of $S$-integers $\co\left[\frac 1m\right]$. In such a setting,
Collinet \cite{C} proved in 2016 that in $\co\left[\frac 12\right]$, every totally positive element is a sum of five squares.
We will first generalize his result to the case of arbitrary even $m$.

Throughout the paper,
$K$ will be a totally real number field of degree $d$ over $\mathbb Q$ with the ring of integers $\mathcal{O}$;
the subset of totally positive elements is denoted $\co^+$.

\begin{thm}\label{thm:1}	
	Let $m\in \mathbb{N}$, $m>1$. 
\begin{enumerate}[label=\alph*)]
	\item If $m$ is even or $2$ is unramified in $\mathcal{O}$, then every totally positive element of
	$\mathcal{O}\left[\frac 1m\right]$ is a sum of five squares in $\mathcal{O}\left[\frac 1m\right]$.
	\item If $m$ is odd and $2$ ramifies in $\mathcal{O}$, then there exist totally positive elements in
	$\mathcal{O}\left[\frac 1m\right]$ that are not sums of squares in $\mathcal{O}\left[\frac 1m\right]$.
\end{enumerate}
\end{thm}

One of our tools is the partial local-global principle for integral representations by quadratic forms 
due to Hsia, Kitaoka, and Kneser \cite{HKK} (although one could instead directly argue via the genus theory used in the proof of their result).
Our resulting Theorem \ref{thm:1} leaves open the question of what happens for arbitrary rings of $S$-integers, although perhaps similar methods may also be useful in the general case.

As a Corollary to Theorem \ref{thm:1} and its proof we get:
\begin{cor}\label{cor:1}
	There exists $m_0\in \mathbb{N}$ (depending on $K$) 
	such that for all $m\geq m_0$,
	every element in $2m\co^+$ is a sum of five squares of elements of $\co$.
\end{cor}

When $2$ is unramified in $K$, then one does not need to consider only even multiples and gets that each element in $m\co^+$ is a sum of five squares of elements of $\co$ if $m$ is sufficiently large.

The Corollary is a significant improvement on the result of Siegel \cite[Satz 9]{S1} which showed that there exists $m\in \mathbb{N}$, depending on $K$, such that any totally positive integer $\alpha \in K$ can be represented as $\sum \left( \dfrac{\beta_i}{m}\right) ^2$. Specifically, the number of squares required by Siegel depended on $\alpha$, i.e., $\dfrac{m^2\Tr(\alpha)}{d}$ squares were needed.

\medskip

We further focus in detail on the case of real quadratic fields $\mathbb Q(\sqrt D)$ in Section \ref{s:3}.

Scharlau \cite{Sc1} showed that in $\mathbb{Q}(\sqrt{2})$ and $\mathbb{Q}(\sqrt{3})$, all the elements of $\co$ that meet local criteria are represented as a sum of squares in $\co$. This implies that all elements of $2\mathcal{O}^+$ are represented as a sum of squares (of elements of $\co$), i.e., that one can take $m_0=1$. 
We prove the converse to his result:

\begin{thm}\label{thm:3}
	Let $K=\mathbb{Q}(\sqrt{D})$ with $D\geq 2$ squarefree.
	Then every element of $2\mathcal{O}^+$ is a sum of squares in $\co$ if and only if $D=2,3$, or $5$.	
\end{thm}

The Pythagoras number of the ring of integers $\co$ in a real quadratic field is at most 5 \cite[Satz 2]{P}, 
so every sum of squares in $\co$ is also a sum of at most five squares. 
In fact, when $D=2,3,5$, the Pythagoras number is 3 \cite[Bemerkung 1]{Sc1}, and so three squares suffice in these cases.

For the proof, we consider the representability of $2\alpha$ for a specific element $\alpha$ (motivated by the study of indecomposable elements \cite{DS,HK}).

It is natural to ask whether for given $m\in \mathbb{N}$ there are only finitely many totally real numbers fields for which every element of $m\co^+$ is a sum of squares in $\co$. At least
for real quadratic fields, we can give the following bounds on $m$:
\begin{thm}\label{cor:2}
	Let $K=\mathbb{Q}(\sqrt{D})$ with $D\geq 2$ squarefree. Let $\kappa=1$ if $D\equiv 1\pmod 4$ and $\kappa=2$ if $D\equiv 2,3\pmod 4$.
\begin{enumerate}[label=\alph*)]
	\item If 
	$m<\frac{\kappa\sqrt D}4$, then \emph{not} all elements of $m\co^+$ are represented as a sum of squares in $\co$.
	\item If $m\geq \frac{D}2$, then all elements of $\kappa m\co^+$ are sums of five squares in~$\co$.
	\item If $m$ is odd and $D\equiv 2,3\pmod 4$, then there exist elements of $m\co^+$ that are not sums of squares in $\co$.
\end{enumerate}
\end{thm}

\section{Local-global principle}\label{s:2}

Let $K$ be a totally real number field of degree $d=[K:\mathbb Q]$ with the ring of integers $\mathcal{O}=\co_K$, i.e., $K$ has $d$ real embeddings $\sigma:K\hookrightarrow\mathbb R$.
An element $\alpha\in K$ is \textit{totally positive} ($\alpha\succ 0$) if $\sigma(\alpha)>0$ for all $\sigma$. For $A\subset K$ we will denote by $A^+$ the subset of totally positive elements of $A$.
The norm and trace from $K$ to $\mathbb Q$ are denoted by $N$ and $\Tr$, respectively.

Some of our arguments will be formulated in the language of quadratic lattices:
Let $F$ be a number field or its completion; if $F\neq \mathbb R,\mathbb C$, then let $\co_F$ be its ring of integers.
A \emph{quadratic space over $F$} is an $r$-dimensional $F$-vector space $V$ equipped with a symmetric bilinear form $B: V\times V\to F$ and the associated quadratic form $Q(v)=B(v,v)$.
A \emph{quadratic $\co_F$-lattice} $L\subset V$ is a finitely generated $\co_F$-submodule such that $FL=V$; $L$ is equipped with the restricted quadratic form $Q$, and so we often talk about a quadratic $\co_F$-lattice as the pair $(L,Q)$.

We will only work with the sum of squares quadratic form $Q=x_1^2+\dots+x_r^2$ and its associated 
quadratic $\co_F$-lattice $(\co_F^r,Q)$. This lattice is always \textit{unimodular}, i.e., 
its \textit{dual} $L^{\#}=\{v\in V\mid B(v,w)\in\co_F\text{ for all }w\in L\}$ equals the lattice $L$.

An element $\alpha$ is \textit{represented} by the $\co_F$-lattice $L$ (or by the $F$-quadratic space $V$) if there is 
$v\in L$ (or $v\in V$) such that $Q(v)=\alpha$.

We denote by $\mathfrak{s}L$ the \textit{scale} of a lattice $L$, i.e., $\mathfrak{s}L=\{\sum B(v_i,w_i)\mid v_i,w_i \in L\}.$ Clearly, $\mathfrak{s}L$ is an ideal in $\mathcal{O}_F$. Let us denote by $Q(L)=\{Q(v) \mid v \in L\}$ the set of all elements represented by $Q$.

\medskip

We shall use the partial local-global principle by Hsia, Kitaoka, and Kneser that we formulate only for a sum of squares, as this is the case we need.

\begin{thm}[{\cite[Theorem 3]{HKK}}]\label{HKK}
	There exists a constant $C=C(K)$ such that, for every $\alpha \in \mathcal{O}^+$ with $N(\alpha)\geq C$, we have that $\alpha$ is represented as a sum of five squares in $\mathcal{O}$ if  and only if $\alpha$ is represented everywhere locally as a sum of five squares.
\end{thm}

To be precise, we say that $\alpha$ is \textit{represented everywhere locally} if it is represented over all the archimedean completions of $K$ and over the rings of integers $\co_F$ in all the non-archimedean completions $F$.
Similarly, when we, e.g., say \textit{over all the completions at dyadic primes}, we mean over the rings of integers $\co_F$ in all the completions $F$ of $K$ at prime ideals $\mathfrak p\mid 2\co$.

\medskip

The next Lemma is a known result (see for example \cite[Remark 1]{BDS}), but we could not find 
its proof in the literature, so we reprove it for the benefit of the reader.

\begin{lemma}\label{lemma-1}
	Let $\alpha\in \mathcal{O}$ and $r\in\mathbb N, r\geq 5$. Then $\alpha$ is a sum of $r$ squares over all the completions at dyadic primes if and only if $\alpha$ is a square modulo $2\mathcal{O}$.
\end{lemma}
\begin{proof}
	
	Let $2\mathcal{O}=\prod^k_{i=1}\mathfrak{p}_i^{e_i}$ where $\mathfrak p_i$ are pairwise distinct prime ideals.
	We have $x_1^2+\dots+x_r^2\equiv (x_1+\dots+x_r)^2\pmod{\mathfrak{p}_i^{e_i}}$, i.e., any sum of squares is congruent to a square modulo $\mathfrak{p}_i^{e_i}\mid 2\co$. 
	Thus, if $\alpha$ is represented as a sum of squares over the completion at $\mathfrak p_i$, then $\alpha$ is congruent to a square modulo $\mathfrak{p}_i^{e_i}$. 
	The Chinese Remainder Theorem then implies that $\alpha$ is congruent to a square also modulo $2\co$, as needed.
	
	\medskip
	
	Now assume that $\alpha\equiv \beta^2 \pmod{2\mathcal{O}}$, i.e.,  $\alpha=\beta^2+2\delta$ for some $\delta \in \mathcal{O}$.
	Consider a prime ideal $\mathfrak p\mid 2\co$ and let $F$ be the corresponding completion of $K$.
	Let $L$ be the $\co_F$-lattice associated to the quadratic form $Q(x)=x^2_1+\ldots + x^2_r$.  
	This sum of squares lattice is unimodular and the rank of $L$ is $r\ge 5$, and so by \cite[93:18(v)]{O} (that uses the assumption that we are working dyadically), we deduce that $Q(L)=Q(L)+2(\mathfrak{s}L)$. At the same time,
	as $Q$ represents 1, we have $\mathfrak{s}L=\mathcal{O}$, and so $\delta \in \mathfrak{s}L$ and $\alpha=\beta^2+2\delta\in Q(L)+2(\mathfrak{s}L)=Q(L)$, as needed.
\end{proof}

%\pagebreak

We can now prove our first Theorem and Corollary.

\begin{customthm}{1}
	Let $m\in \mathbb{N}$, $m>1$. 
	\begin{enumerate}[label=\alph*)]
		\item If $m$ is even or $2$ is unramified in $\mathcal{O}$, then every totally positive element of
		$\mathcal{O}\left[\frac 1m\right]$ is a sum of five squares in $\mathcal{O}\left[\frac 1m\right]$.
		\item If $m$ is odd and $2$ ramifies in $\mathcal{O}$, then there exist totally positive elements in
		$\mathcal{O}\left[\frac 1m\right]$ that are not sums of squares in $\mathcal{O}\left[\frac 1m\right]$.
	\end{enumerate}
\end{customthm}

\begin{proof}
	
a) 
Let us consider a totally positive element in $\mathcal{O}\left[\frac 1m\right]$ and write it as $\frac{\alpha}{m^k}$, where $\alpha\in \mathcal{O}$ and $k\in\mathbb{Z}$. By multiplying the numerator and denominator of $\frac{\alpha}{m^k}$ by a sufficiently large power of $m$, we can without loss of generality assume that $\alpha$ satisfies the norm assumption in Theorem \ref{HKK} and that $k$ is even. Thus to apply this Theorem, we need to show that $\alpha$ is represented everywhere locally as a sum of five squares. This is clear over all archimedean places, as $\alpha$ is totally positive.
	
Over every non-dyadic prime, every unimodular lattice of rank $\ge 3$ is universal \cite[92:1b]{O}, and thus $\alpha$ is a sum of $3$ squares (in the considered non-dyadic completion of $K$), and \textit{a fortiori} of $5$ squares. 

For a dyadic prime $\mathfrak{p}\mid 2\mathcal{O}$, by Lemma \ref{lemma-1}, $\alpha$ is a sum of five squares if and only if $\alpha$ is a square modulo $2\mathcal{O}$. If $m$ is even, then we can without loss of generality replace $\frac{\alpha}{m^k}$ by $\frac{m^2\alpha}{m^{k+2}}$ so that $\alpha\equiv 0\equiv 0^2 \pmod{2\mathcal{O}}$.
If $2$ is unramified in $\mathcal{O}$, then $\mathcal{O}/2\mathcal{O}$ is a product of fields of characteristic 2, and hence every element of $\co$ is a square modulo $2\co$.

\medskip

b)
Assume that $m$ is odd and $2$ ramifies, i.e.,  $2\mathcal{O}=\prod^k_{i=1}\mathfrak{p}_i^{e_i}$, where $\mathfrak p_i$ are pairwise distinct primes and at least one $e_i>1$; assume that $e_1>1$. 
Choose an element  $\alpha_0\in\mathfrak{p}_1\setminus \mathfrak{p}_1^2$ and add to it a sufficiently large even rational integer $2a$ so that the sum $\alpha=\alpha_0+2a\in\mathfrak{p}_1\setminus \mathfrak{p}_1^2$ is totally positive;
assume for contradiction that $\alpha$ is a sum of squares in $\mathcal{O}\left[\frac 1m\right]$.

As $m$ is odd, it is invertible modulo $\mathfrak p_1^{e_1}$, and so $\alpha$ is congruent to a sum of squares in $\co/\mathfrak p_1^{e_1}$. As in the first line of the proof of Lemma \ref{lemma-1}, this implies that
$\alpha\equiv \beta^2\pmod{\mathfrak{p}_1^{e_1}}$ for some $\beta\in\co$.
We have $\alpha \in \mathfrak{p}_1$, and so $\beta^2 \in \mathfrak{p}_1$. Therefore $\beta\in\mathfrak{p}_1$ (for $\mathfrak{p}_1$ is a prime ideal), and finally $\alpha\in \mathfrak{p}_1^2$, a contradiction. 
\end{proof}

\begin{customcor}{2}
	There exists $m_0\in \mathbb{N}$ (depending on $K$) such that for all $m\geq m_0$,
	every element in $2m\co^+$ is a sum of five squares of elements of $\co$.
\end{customcor}

\begin{proof}
Let $2m\geq C^{1/d}$ where $C$ is the constant from Theorem \ref{HKK} and consider $\alpha\in\co^+$. We need to show that $2m\alpha$ can be represented as a sum of five squares in $\co$.

We have $N(2m\alpha)\geq (2m)^{d}\geq C$, and so the norm condition from Theorem \ref{HKK} is satisfied for $2m\alpha$.

Further, $2\mid 2m\alpha$, and so the element $2m\alpha$ is represented everywhere locally as we saw in the proof of Theorem~\ref{thm:1}. Thus $2m\alpha$ is a sum of five squares by Theorem \ref{HKK}.
\end{proof}

%\pagebreak

In the case when $m$ is odd and $2$ ramifies, one can argue analogously using Theorem \ref{HKK} (and the observation that a sum of squares always satisfies the local conditions at dyadic primes) to show that the Pythagoras number of $\co\left[\frac 1m\right]$ is at most 5.

\section{Quadratic case}\label{s:3}

To prove  Theorem \ref{thm:3}, we will need to work with specific elements of $K$. Throughout this Section, we  use the following notation. We fix a squarefree integer $D\geq2$ and
consider the real quadratic field $K=\mathbb Q(\sqrt D)$ and its ring of integers $\co$.
Then $\{1,\omega_D\}$ forms an integral basis of $\co$, where
\[
\omega_D= \begin{cases}
\sqrt D & \text{if }\textstyle D\equiv 2,3\pmod 4
,\\
\frac{1+\sqrt D}{2} & \text{if }D\equiv 1\pmod 4
.\end{cases}
\]
The conjugate of $\alpha\in K$ is $\alpha'$.

\begin{customthm}{3}
	Let $K=\mathbb{Q}(\sqrt{D})$ with $D\geq 2$ squarefree.
	Then every element of $2\mathcal{O}^+$ is a sum of squares in $\co$ if and only if $D=2,3$, or $5$.		
\end{customthm}

\begin{proof}
	
	``$\Rightarrow$'' 	
	Assume that every element of $2\co^+$ is a sum of squares.
	
	Let us start with the case $D\equiv 2,3\pmod 4$.
	
	For the element
	$\alpha=k+\sqrt D\in\co^+$ with $k=\lfloor \sqrt D\rfloor +1$, let
	$2\alpha=\sum_i\alpha_i^2$ be the corresponding decomposition with
	$\alpha_i=a_i+b_i\sqrt D$, where we can without loss of generality assume that $a_i\geq 0$.
	We have	$$\alpha_i^2= a_i^2+Db_i^2+2a_ib_i\sqrt D.$$
	
	If $a_i=0$, then let us also without loss of generality assume that $b_i>0$.
	
	If for some $i$ we have $a_i>0$ and $b_i<0$, then $2\alpha'\geq (\alpha_i')^2= a_i^2+Db_i^2+2a_i(-b_i)\sqrt D\geq 1+D+2\sqrt D>2$, a contradiction, as $\alpha'<1$.
	Thus all $b_i\geq 0$.
	
	Comparing irrational parts in the expression for $2\alpha$, we see that $1 =\sum_i a_ib_i$. 
	All the summands are non-negative, and so there is exactly one index $j$ with $a_jb_j=1$ and $a_ib_i=0$ for all $i\neq j$.
	Without loss of generality let $j=1$ so that $a_1=b_1=1$.
	
	Subtracting $\alpha_1^2$ from the representation of $2\alpha$ we then get $2k-(1+D)=\sum_{i\geq 2}\alpha_i^2\geq 0$, and so
	\[2\sqrt D+2>2k\geq 1+D.\]
	This is possible only for $D<6$ (i.e., $D=2,3$)  finishing this part of the proof.
	
	\medskip
	
	Let us now take $D\equiv 1\pmod 4$ and 
	$\alpha=k+\omega_D=(k+1/2)+\sqrt D/2\in\co^+$ with $k=\lfloor \omega_D\rfloor$. Again, let 
		$2\alpha=\sum_i\alpha_i^2$ with
	$\alpha_i=a_i+b_i\sqrt D$ and $a_i\geq 0$, except this time $a_i=c_i/2, b_i=d_i/2$ can be half-integers (with $2\mid c_i-d_i$ for all $i$). 
	Nevertheless, as before we have $\alpha'<1$, from which we deduce that  $b_i\geq 0$ for all $i$.
	
	Comparing coefficients we then have $1/2=\sum_i a_ib_i$, and so (without loss of generality) the only possibility is $\alpha_1=\alpha_2=\omega_D$ and $a_ib_i=0$ for all $i>2$. Then $2k+1/2-D/2=2\alpha-2\omega_D^2=\sum_{i\geq 3}\alpha_i^2\geq 0$. Thus
	\[2\sqrt D+3\geq D,\]
	which is not possible for $D>9$. Thus $D=5$ in this case. 
		
	\medskip
	
	``$\Leftarrow$'' When $D=5$, Maa\ss\ \cite{Ma} showed that every element of $\co^+$ is a sum of three squares.
	
	When $D=2,3$, Scharlau \cite{Sc1} in fact proved that each element that is congruent to a square modulo $2\co$ is a sum of squares. In particular, all elements of $2\co^+$ are sums of squares.
\end{proof}

Although we did not need this explicitly in the proof, its idea is based on the study of indecomposable integers.
We say that $\alpha\in\co^+$ is \emph{indecomposable} if it can not be written
as a sum of two totally positive integers. 
These elements can be explicitly characterized in terms of the periodic continued fraction for $\omega_D$ \cite{DS};
our proof was motivated by the study of the possible decompositions of $2\alpha$ \cite{HK}.

Note that it should be possible to use a similar method based on explicitly studying representations of multiples of indecomposables to determine all the real quadratic fields where all elements of $3\co^+$ (or $4\co^+$ or $m\co^+$ for other fixed $m$) are sums of squares. However, the technical difficulty quickly increases with $m$. Nevertheless we at least get

\begin{customthm}{4}
	Let $K=\mathbb{Q}(\sqrt{D})$ with $D\geq 2$ squarefree. Let $\kappa=1$ if $D\equiv 1\pmod 4$ and $\kappa=2$ if $D\equiv 2,3\pmod 4$.
	\begin{enumerate}[label=\alph*)]
		\item If 
		$m<\frac{\kappa\sqrt D}4$, then \emph{not} all elements of $m\co^+$ are represented as a sum of squares in $\co$.
		\item If $m\geq \frac{D}2$, then all elements of $\kappa m\co^+$ are sums of five squares in~$\co$.
		\item If $m$ is odd and $D\equiv 2,3\pmod 4$, then there exist elements of $m\co^+$ that are not sums of squares in $\co$.
	\end{enumerate}
\end{customthm}

\begin{proof}
a)	The argument is similar to the proof of Theorem \ref{thm:3}, so let us skip some details. 

If $D\equiv 2,3\pmod 4$, let $\alpha=k+\sqrt D$ (with $k=\lfloor \sqrt D\rfloor +1$)
be the same element as in the proof of Theorem \ref{thm:3} and let $m\alpha=\sum_i (a_i+b_i\sqrt D)^2$ with $a_i,b_i\in\mathbb Z$. As before we obtain that $a_i\geq 0, b_i\geq 0$ for all $i$ (using the upper bound on $m$).

Comparing irrational parts in the expression for $m\alpha$ we see that $m=2\sum_i a_ib_i$, and so for some $j$ we have $a_j,b_j\geq 1$. Then 
$m(2\sqrt D+1)> m\alpha\geq(a_j+b_j\sqrt D)^2\geq (1+\sqrt D)^2$, which gives a contradiction with $m<\frac{\sqrt D}{2}$.

If $D\equiv 1\pmod 4$, we consider the element $\alpha=k+\omega_D$ with $k=\lfloor \omega_D\rfloor$. If $m\alpha=\sum_i (a_i+b_i\sqrt D)^2$ with $a_i,b_i\in\mathbb Z/2$, then again $a_i\geq 0, b_i\geq 0$ for all $i$, and so $a_j,b_j\geq 1/2$ for some $j$. 
This finally implies  $m(\sqrt D+1)> m\alpha\geq(a_j+b_j\sqrt D)^2\geq(1/2+\sqrt D/2)^2$, which is again impossible.

\medskip

b)	The proof will use a result of Peters \cite[Satz 2]{P}. To state it, assume first that $D\equiv 1 \pmod{4}$. Given $\alpha \in \mathcal{O}^+$ of the form $\alpha=a_0+a_1\omega_D$, Peters' theorem states that $\alpha$ is a sum of five squares if and only if there exists a rational integer $n$ in the interval 
	\[\left[\dfrac{1}{D}\left(2a_0+a_1-2\sqrt{N(\alpha)}\right),\dfrac{1}{D}\left(2a_0+a_1+2\sqrt{N(\alpha)}\right)\right]\] 
	such that $n\equiv a_1 \pmod{2}$. 
	
	If the length of this closed interval is at least 2, then it contains at least two consecutive integers, one of which will satisfy the requirement modulo~2. Therefore, if 
	\[2\sqrt{N(\alpha)}\ge D,\]
	then $\alpha$ is represented as a sum of five squares. 
	
	Every element $\alpha=m\beta$ with $\beta\in\co^+$ and
	$m\geq D/2$ satisfies the norm inequality above.
	
	\medskip
	
	In the case of $D\equiv 2,3\pmod 4$, let $\alpha=a_0+2a_1\sqrt{D}$. Then by the same result of Peters, $\alpha$ is a sum of five squares if there exists a rational integer $n$ in 
	\[\left[\dfrac{1}{2D}\left(a_0-\sqrt{N(\alpha)}\right), \dfrac{1}{2D}\left(a_0+\sqrt{N(\alpha)}\right) \right].\] 
	The length of this interval is at least 1 when $\alpha=2m\beta$ with $\beta\in\co^+$ and
	$m\geq \frac D2$, and the coefficient of $\alpha=2m\beta$ at $\sqrt D$ is obviously even as required.

\medskip
	
c) Let $D\equiv 2,3\pmod 4$ and consider the element $\alpha=k+\sqrt D$ from part a). If $m\alpha$ is a sum of squares as in part a), then $m=2\sum_i a_ib_i$ is even.
\end{proof}

Obviously, the argument in part a) and the bound on $m$ can be somewhat improved by considering also the other summands $(a_i+b_i\sqrt D)^2$.

At least at present, this method is only limited to the quadratic case, as in higher degrees we lack the necessary understanding of indecomposables.

\medskip

Note that when $D\equiv 1 \pmod{4}$, then Peters' result immediately implies that every element $\alpha\in\co^+$ of norm $>D/2$ is represented as a sum of five squares. E.g., when $D=5$, then all totally positive units are squares and the smallest norm of a non-unit is $3>5/2$, and so this implies that all elements of $\co^+$ are sums of squares.

\section*{Acknowledgments}
We thank the anonymous referee for a very careful and fast reading of the manuscript, and for a number of detailed comments that helped to fix some issues and to improve the readability of the article. We also thank Martin Ra\v ska for pointing out an error in Theorem 4b).

\end{document}